\documentclass[12pt]{amsart}
\usepackage{amsfonts,amsmath,amscd,amssymb,amsthm,paralist}
\usepackage{mathrsfs}
\usepackage{amsxtra}
\usepackage[mathscr]{eucal}
\usepackage{graphics}
\usepackage{parskip}
\usepackage{latexsym,graphicx,verbatim,nameref}
\usepackage[numbers,sort&compress]{natbib}
\usepackage[all,cmtip]{xy}
\usepackage{hyperref}
\usepackage{fancyhdr}
\usepackage{color}
\usepackage{tikz}
\usepackage{tikz-cd}
\usetikzlibrary{arrows,scopes,matrix}
\theoremstyle{plain}

\theoremstyle{definition}

\makeatletter
\def\thm@space@setup{%
  \thm@preskip=\parskip \thm@postskip=0pt
}
\makeatother

\makeatletter
\newtheorem*{rep@theorem}{\rep@title}
\newcommand{\newreptheorem}[2]{%
\newenvironment{rep#1}[1]{%
 \def\rep@title{#2 \ref{##1}}%
 \begin{rep@theorem}}%
 {\end{rep@theorem}}}
\makeatother

\newreptheorem{theorem}{Theorem}
\newreptheorem{corollary}{Corollary}
\newreptheorem{lemma}{Lemma}

\newtheorem{theorem}{Theorem}[section]
\newtheorem{corollary}[theorem]{Corollary}
\newtheorem{lemma}[theorem]{Lemma}
\newtheorem{proposition}[theorem]{Proposition}
\newtheorem*{theorem*}{Theorem}
\newtheorem*{corollary*}{Corollary}
\newtheorem*{proposition*}{Proposition}

\newtheorem{definition}[theorem]{Definition}
\newtheorem{remark}[theorem]{Remark}

\newtheorem*{claim*}{Claim}

\newtheorem*{conjecture*}{Conjecture}
\newtheorem*{observation*}{Observation}
\newtheorem*{question*}{Question}

\numberwithin{equation}{section}

\newcommand{\ub}{{\rm U}}

\newcommand{\fai}{\varphi}

\begin{document}
\title{Ergodic invariant states and irreducible representations of crossed product $C^*$-algebras}
\author{Huichi Huang}
\email{huanghuichi@gmail.com}
\author{Jianchao Wu}
\address{Jianchao Wu, Mathematisches Institut, Universit\"{a}t M\"{u}nster, Einsteinstr. 62,  M\"unster, 48149, Germany}
\email{jianchao.wu@uni-muenster.de}
\date{\today}
\begin{abstract}
Motivated by reformulating  Furstenberg's $\times p,\times q$ conjecture via representations of a crossed product $C^*$-algebra,  we show that in a discrete  $C^*$-dynamical system $(A,\Gamma)$, the space of (ergodic) $\Gamma$-invariant states on $A$ is homeomorphic to a subspace of (pure) state space of $A\rtimes\Gamma$. Various applications of this in topological dynamical systems and representation theory are obtained. In particular, we prove that the classification of ergodic $\Gamma$-invariant regular Borel probability measures on a compact Hausdorff space $X$ is equivalent to the classification  a special type of irreducible representations of $C(X)\rtimes \Gamma$.
\end{abstract}

\maketitle

\section{Introduction}\

Assume  that $p,q$ are two positive integers greater than 1 with $\frac{\log{p}}{\log{q}}$ irrational. 

H. Furstenberg gives the classification of closed $\times p,\times q$-invariant subsets of the unit circle $\mathbb{T}$, which says such a set is either finite or $\mathbb{T}$~\cite[Theorem IV.1.]{Furstenberg1967}. He also gives the following conjecture concerning the classification of  ergodic $\times p,\times q$-invariant measures on $\mathbb{T}$.

\begin{conjecture*}~[Furstenberg's $\times p,\times q$ conjecture]\

An ergodic $\times p,\times q$-invariant  Borel probability measure on $\mathbb{T}$ is either finitely supported or the Lebesgue measure.
\end{conjecture*}

Furstenberg's conjecture is the simplest case of conjectures concerning classifications of invariant measures, and there are vast literatures about its general versions and their applications in number theory. See~\cite{EinsiedlerLindenstrauss2006} for a survey.

For Furstenberg's conjecture, the best known result is the following theorem, which is proven by D. J. Rudolph under the assumption that $p,q$ is coprime~\cite[Theorem 4.9.]{Rudolph1990}, later improved by A. S. A. Johnson~\cite[Theorem A]{Johnson1992}.

\begin{theorem*}~[Rudolph-Johnson's Theorem]\

If $\mu$ is an ergodic $\times p,\times q$-invariant measure on $\mathbb{T}$, then either $h_\mu(T_p)=h_\mu(T_q)=0$ or $\mu$ is the Lebesgue measure.
\end{theorem*}
Here $h_\mu(T_p)$ and $h_\mu(T_q)$ stand for the measure-theoretic entropy of $\times p$ and $\times q$ with  respect to $\mu$  respectively. See~\cite[Chapter 4]{Walters1982} for the definition of entropy for measure preserving maps.

For a $\times p,\times q$-invariant measure on $\mathbb{T}$, denote the two isometries on $L^2(\mathbb{T},\mu)$ induced by continuous maps $\times p,\times q:\mathbb{T}\to\mathbb{T}$ by $V_p, V_q$.

By Rudolph-Johnson's Theorem,  to classify ergodic $\times p,\times q$-invariant measures on $\mathbb{T}$, it suffices to classify such ergodic measures with zero entropy for $T_p$ or $T_q$.

J. Cuntz notices that when $h_\mu(T_p)=h_\mu(T_q)=0$, the operators $V_p$ and $V_q$ are two commuting unitary operators on $L^2(\mathbb{T},\mu)$~\cite[Corollary 4.14.3]{Walters1982}.

For the unitary operator $M_z:L^2(\mathbb{T},\mu)\to L^2(\mathbb{T},\mu)$ given by $M_zf(z)=zf(z)$ for all $f\in L^2(\mathbb{T},\mu)$ and $z\in \mathbb{T}$, one have $V_p M_z=M_z^pV_p$ and $V_q M_z=M_z^qV_q$. So a $\times p,\times q$ invariant measure $\mu$ with zero entropy gives rise to a representation $\pi_\mu$ of the universal unital $C^*$-algebra $C^*(s,t,z)$ generated by three unitaries $s,t$ and $z$ with relations
$$st=ts,\quad sz=z^ps,\quad tz=z^qt$$ in the following way:
$$\pi_\mu(s)=V_p,\quad \pi_\mu(t)=V_q, \quad \pi_\mu(z)=M_z.$$

With the above observation,  Cuntz suggests that one can consider ergodic $\times p,\times q$-invariant measures on $\mathbb{T}$ via  representations of $C^*(s,t,z)\cong C^*(\mathbb{Z}[\frac{1}{pq}])\rtimes\mathbb{Z}^2$, where the two generators of $\mathbb{Z}^2$ acts on $C^*(\mathbb{Z}[\frac{1}{pq}])$ by automorphisms induced by $\times p,\times q$ maps on $[\frac{1}{pq}]$, and the isomorphism
$\Phi:C^*(s,t,z)\to C^*(\mathbb{Z}[\frac{1}{pq}])\rtimes\mathbb{Z}^2$ is given by $\Phi(s)=a, \Phi(t)=b$ and $\Phi(z)=1$. Here  $a=(1,0)$ and $b=(0,1)$ are in $\mathbb{Z}^2$ and $1$ is in $\mathbb{Z}[\frac{1}{pq}]$~\cite{Cuntz2013}.

Motivated by Cuntz's observation, firstly one have to answer the following question:

{\it what kind of representation of $C^*(\mathbb{Z}[\frac{1}{pq}])\rtimes\mathbb{Z}^2$ is induced by a $\times p,\times q$-invariant measure on $\mathbb{T}$?}

Denote the dual of $\mathbb{Z}[\frac{1}{pq}]$ by $S_{pq}$, the $pq$-solenoid~\cite[A.1]{Robert2000}. The $\times p,\times q$ isomorphisms on $\mathbb{Z}[\frac{1}{pq}]$ give rise to $\times p,\times q$ isomorphisms on $S_{pq}$.

We answer the above question in the following way.

Firstly the space of ergodic $\times p,\times q$-invariant measures on $\mathbb{T}$ is homeomorphic to the space of ergodic $\times p,\times q$-invariant measures on $S_{pq}$, hence the classification of ergodic $\times p,\times q$-invariant measures on $\mathbb{T}$ amounts to classification of  ergodic $\times p,\times q$-invariant measures on $S_{pq}$. Secondly  ergodic $\times p,\times q$-invariant measures on $S_{pq}$ 1-1 corresponds to irreducible representations of $C^*(\mathbb{Z}[\frac{1}{pq}])\rtimes\mathbb{Z}^2$ whose restriction to $\mathbb{Z}^2$ contains the trivial representation.

Moreover in a more general context, we prove the following which briefly shows how the problem of invariant states relates to crossed product $C^*$-algebras.

Assume that a discrete group $\Gamma$ acts on a unital $C^*$-algebra $A$ as automorphisms. Denote this action by $\alpha$, which is, a  group homomorphism from $\Gamma$ to the automorphism group $Aut(A)$ of $A$.

A state $\fai$ on $A$ is {\bf $\Gamma$-invariant} if $\fai(\alpha_s(a))=\fai(a)$ for all $s$ in
$\Gamma$ and $a$ in $A$.  An extreme point of the set of $\Gamma$-invariant states on $A$~\footnote{This is a closed convex set when equipped with weak-$*$ topology, hence when nonempty, the set of extreme points is also nonempty.}  are called {\bf ergodic}.

Denote by $A\rtimes\Gamma$ the full crossed product of the $C^*$-dynamical system $(A,\Gamma,\alpha)$.

\begin{reptheorem}{homeomorphism between states}\
 The space of (ergodic) $\Gamma$-invariant states on $A$ is homeomorphic to the space of (pure) states on $A\rtimes\Gamma$ whose restriction to $\Gamma$ is the trivial character.
\end{reptheorem}

We give some applications of Theorem~\ref{homeomorphism between states} to topological dynamical systems and representation theory.

Suppose a discrete group $\Gamma$ acts on a compact Hausdorff space $X$ as homeomorphisms~(this is the same as $\Gamma$ acting on the unital $C^*$-algebra $C(X)$, the space of continuous functions on $X$, by automorphisms). For a representation $\pi: C(X)\rtimes\Gamma\to B(H)$, denote the space of $\Gamma$-invariant vectors in $H$ by $H_\Gamma$.

\begin{reptheorem}{dim}\
Every irreducible representation  $\pi$ of $C(X)\rtimes\Gamma$ on a Hilbert space $H$ satisfies that $\dim{H_\Gamma}\leq 1$. When $\dim{H_\Gamma}=1$, the representation $\pi$ is uniquely induced by an ergodic $\Gamma$-invariant regular Borel probability measure $\mu$ on $X$.
\end{reptheorem}

A special case of Theorem~\ref{dim} is the following.

\begin{repcorollary}{rep}\
Suppose a discrete group $\Gamma$ acts on a discrete abelian group $G$ by group automorphisms.

Every irreducible unitary representation $\pi$ of $G\rtimes\Gamma$ on a Hilbert space $H$ satisfies that $\dim{H_\Gamma}\leq 1$.

When $\dim{H_\Gamma}=1$, the representation $\pi$ is uniquely induced by an ergodic $\Gamma$-invariant regular Borel probability measure $\mu$ on the Pontryagin dual $\widehat{G}$ of $G$.
\end{repcorollary}

The paper is organized as follows.

In the preliminary section, we recall some background of crossed product $C^*$-algebras.  The proof of  Theorem~\ref{homeomorphism between states} is given in section 3.2. At the end of this section, we include two immediate applications of Theorem~\ref{homeomorphism between states}, namely, Proposition~\ref{ergodic via projections} and Proposition~\ref{nonempty}, to $C^*$-dynamical systems. In section 3.3, we show Theorem~\ref{ErgodicIrr} and Theorem~\ref{dim}. In the last section we prove Theorem~\ref{finiterepresentation} which enables us to reformulate Furstenberg's $\times p,\times q$ problem in terms of representation theory of the semidirect product group $\mathbb{Z}[\frac{1}{pq}]\rtimes\mathbb{Z}^2$.


\section*{Acknowledgements}

We are grateful to Joachim Cuntz for sharing his insight for Furstenberg's problem with us. We benefits a lot from various discussions with him. H. Huang would thank Xin Li for his suggestion to consider invariant measures in  more general settings. He also thanks  Hanfeng Li for his detailed comments. Thanks are also due to Kang Li for his pointing out the reference~\cite{WillettYu2014} to us which helps to prove Proposition~\ref{nonempty}. We also thank Sven Raum for his valuable comments which lead to much briefer proofs of Lemma~\ref{ResInv} and Proposition~\ref{ergodic via projections}. We also thank an anonymous referee for helpful comments.

The paper was finished when we were postdoctoral fellows supported by  ERC Advanced Grant No. 267079.

\section{Preliminary}\

In this section, we list some background for $C^*$-dynamical systems.

Within this article $\Gamma$ stands for a discrete group and $A$ stands for a unital $C^*$-algebra whose state space and pure state space are denoted by $S(A)$ and $P(A)$ respectively.

Denote the GNS representation of $A$ with respect to a $\fai\in S(A)$ by $\pi_\fai: A\to B(L^2(A,\fai))$ where $L^2(A,\fai)$ stands for the Hilbert space corresponding to $\pi_\fai$. Let
$I_\fai=\{a\in A|\,\fai(a^*a)=0\}$. Denote $a+I_\fai$ by $\hat{a}$ for all $a\in A$.

\begin{definition}~\label{invstate}
An action of $\Gamma$ on $A$ as automorphisms is a group homomorphism $\alpha: \Gamma\to Aut(A)$, where $Aut(A)$ stands for the set of $*$-isomorphisms from $A$ to $A$~(this is a group under composition). We call  $(A,\Gamma,\alpha)$ a dynamical system.

A {\bf $\Gamma$-invariant state} is a state $\fai$ on $A$ such that $\fai(\alpha_s(a))=\fai(a)$ for all $s\in\Gamma$ and $a\in A$~\cite{Ruelle1966}. Denote the set of $\Gamma$-invariant states on $A$ by $S_\Gamma(A)$. It is clear that $S_\Gamma(A)$ is a convex closed set under weak-$*$ topology. If  $S_\Gamma(A)$ is nonempty, then it contains at least one extreme point. We call an extreme point of $S_\Gamma(A)$ an {\bf ergodic $\Gamma$-invariant state} on $A$. The set of ergodic $\Gamma$-invariant states on $A$ is denoted by $E_\Gamma(A)$.
\end{definition}

A representation of a $C^*$-algebra $B$ on a Hilbert space $H$ is a $*$-homomorphism $\pi:B\to B(H)$ and it is called irreducible if the commutant $C(\pi(B))$ consisting of elements in $B(H)$ commuting with every element in $\pi(B)$ are scalar multiples of identity operator.

A {\bf covariant} representation $(\pi,\ub,H)$ of a dynamical system  $(A,\Gamma,\alpha)$ consists of a representation $\pi$ of $A$ and a unitary representation $\ub$ of $\Gamma$ on a Hilbert space $H$ such that
$$\pi(\alpha_s(a))=\ub_s\pi(a)\ub_s^*$$ for all $a\in A$ and $s\in\Gamma$.

Let $C_c(\Gamma,A)$ be the space of finitely supported $A$-valued functions on $\Gamma$. For $f,g\in C_c(\Gamma,A)$,  the product $f*g$ is given by
 $$f*g(t)=\sum_{s_1s_2=t}f(s_1)\alpha_{s_1}(g(s_2))$$  and $f^*$ is given by $$f^*(t)=\alpha_t(f(t^{-1})^*)$$ for every $t\in\Gamma$. Then $C_c(\Gamma,A)$ is a $*$-algebra . Given a covariant representation  $(\pi,\ub,H)$ of a dynamical system  $(A,\Gamma,\alpha)$,  one can construct a representation of $C_c(\Gamma,A)$ on $H$.

\begin{definition}~\label{crossed product}
For a dynamical system $(A,\Gamma,\alpha)$, the {\bf crossed product $C^*$-algebra} $A\rtimes\Gamma$ is the completion of $C_c(\Gamma,A)$ under the norm $\|f\|=\sup\|\tilde{\pi}(f)\|$ for $f\in C_c(\Gamma,A)$ where the supreme is taken over all representations of $C_c(\Gamma,A)$. Denote by $u_s$ the unitary in $A\rtimes\Gamma$ corresponding to an $s\in\Gamma$.
\end{definition}

There is a one-one correspondence between  representations of $A\rtimes\Gamma$ and covariant representations of $(A,\Gamma,\alpha)$.

We refer readers to~\cite[Chapter VIII]{Davidson1996} for more about discrete crossed products.

\section{Main results}\

\subsection{Covariant representations of $(A,\Gamma,\alpha)$ induced by invariant states}\

If $\fai\in S_\Gamma(A)$, then there is a unitary representation~(the Koopman representation) $\ub_\fai$ of $\Gamma$ on $L^2(A,\fai)$ given by
$$\ub_\fai(s)(\hat{a})=\widehat{\alpha_s(a)}$$ for all $s\in\Gamma$ and $a\in A$~\cite{Ruelle1966,LanfordRuelle1967,RobinsonRuelle1967}.  We give  details below for completeness.

The representation $\ub_\fai$ is  well-defined since

\begin{enumerate}
\item For every $a\in I_\fai$, we have $$\fai(\alpha_s(a)^*\alpha_s(a))=\fai(\alpha_s(a^*a))=\fai(a^*a)=0$$ for all $s\in\Gamma$.
\item For every $s\in\Gamma$, the map $\ub_\fai(s)$ is surjective since its image is dense in $L^2(A,\fai)$, and $\ub_\fai(s)$ is an isometry since for all $a\in A$,
$$\langle \widehat{\alpha_s(a)}, \widehat{\alpha_s(a)}\rangle=\fai(\alpha_s(a)^*\alpha_s(a))=\fai(\alpha_s(a^*a))=\fai(a^*a)=\langle \hat{a},\hat{a}\rangle.$$ So $\ub_\fai(s)$ is a unitary.
\item $\ub_\fai(st)(\hat{a})=\widehat{\alpha_{st}(a)}=\ub_s(\ub_t(\hat{a}))$ for all $s,t\in\Gamma$ and $a\in A$.
\end{enumerate}

Hence $\ub_\fai$ is a unitary representation of $\Gamma$ on $L^2(A,\fai)$.

Furthermore we have the following.
\begin{lemma}~\label{covariant}
Given a $\Gamma$-invariant state $\fai$ on $A$, the triple $(\pi_\fai,\ub_\fai,L^2(A,\fai))$ gives a covariant representation of $(A,\Gamma,\alpha)$. So there is a representation of $A\rtimes\Gamma$ on $L^2(A,\fai)$, which we denote by $\rho_\fai$,  given by $$\rho_\fai(\sum_{s\in\Gamma}a_su_s)=\sum_{s\in\Gamma}\pi_\fai(a_s)\ub_\fai(s)$$ for any $\sum_{s\in\Gamma}a_su_s\in C_c(\Gamma,A)$.
\end{lemma}
\begin{proof}
For all $a,b\in A$ and $s\in\Gamma$, we have
\begin{align*}
&\ub_\fai(s)(\pi_\fai(a))\ub_\fai(s^{-1})(\hat{b})=\ub_\fai(s)(\pi_\fai(a))((\alpha_{s^{-1}}(b))^\wedge) \\
&=\ub_\fai(s)((a\alpha_{s^{-1}}(b))^\wedge)=\alpha_s((a\alpha_{s^{-1}}(b))^\wedge)=(\alpha_s(a)\alpha_{ss^{-1}}(b))^\wedge       \\
&=(\alpha_s(a)b)^\wedge=\pi_\fai(\alpha_s(a))(\hat{b}).
\end{align*}
This completes the proof.
\end{proof}

\subsection{$\Gamma$-invariant states on $A$ and states on $A\rtimes\Gamma$}\

Denote $\{\fai\in S(A\rtimes\Gamma)\,|\,\fai(u_s)=1 \, \rm{for \,all} \,s\in\Gamma\}$ by $S^1(A\rtimes\Gamma)$ and $\{\psi\in P(A\rtimes\Gamma)|\psi(u_s)=1\, \rm{for\, all}\, s\in\Gamma\}$ by $P^1(A\rtimes\Gamma)$.

We have the following.

\begin{theorem}~\label{homeomorphism between states}
When equipped with  weak-$*$ topologies, the restriction maps $R: S_\Gamma(A)\to S^1(A\rtimes\Gamma)$ and $R: E_\Gamma(A)\to P^1(A\rtimes\Gamma)$ are homoemorphisms.
\end{theorem}

To prove this theorem, we first prove the following lemma which says that for every $\fai$ in $S^1(A\rtimes\Gamma)$, the restriction $\fai|_A$ belongs to $S_\Gamma(A)$.

\begin{lemma}~\label{ResInv}
For any state $\fai$ on $A\rtimes\Gamma$ such that $\fai(u_s)=1$ for every $s\in\Gamma$, we have $\fai(u_s au_t)=\fai(a)$ for all $a\in A$ and $s,t\in\Gamma$. Consequently the restriction $\fai|_A$  is a $\Gamma$-invariant state  on $A$.
\end{lemma}
\begin{proof}
 The proof  follows from ~\cite[Proposition 1.5.7.]{BrownOzawa2008} since by assumption  every $u_s$ is contained in the multiplicative domain of $\fai$.

For convenience of readers, we give a direct proof which is also based on Cauchy-Schwarz inequality like~\cite[Proposition 1.5.7.]{BrownOzawa2008}.

For all $a\in A$ and $s,t\in\Gamma$, we have
\begin{align*}
&|\fai(u_s au_t)-\fai(a)|\leq |\fai(u_s au_t)-\fai(au_t)|+|\fai(au_t)-\fai(a)| \\
&=|\fai((u_s-1) au_t)|+|\fai(a(u_t-1))|    \\
&=|\langle au_t, (u_s-1)^*\rangle_{L^2(A\rtimes\Gamma,\fai)}|+|\langle u_t-1, a^*\rangle_{L^2(A\rtimes\Gamma,\fai)}| \\
\tag{Cauchy-Schwarz Inequality}
&\leq [\fai((u_t)^*a^*au_t)]^{\frac{1}{2}}[\fai((u_s-1)(u_s-1)^*)]^{\frac{1}{2}}+[\fai((u_t-1)^*(u_t-1))]^{\frac{1}{2}}[\fai(aa^*)]^{\frac{1}{2}} \\
\tag{$\fai(u_t)=1$ implies that $\fai((u_t-1)(u_t-1)^*)=0$ and $\fai((u_t-1)^*(u_t-1))=0$.}
&=0.
\end{align*}

\end{proof}

It follows from Lemma~\ref{covariant} that for a $\Gamma$-invariant state $\fai$, there is a representation $\rho_\fai$ of $A\rtimes\Gamma$ on $L^2(A,\fai)$ given by
$$\rho_\fai(\sum_{s\in\Gamma}a_su_s)=\sum_{s\in\Gamma}\pi_\fai(a_s)\ub_\fai(s)$$ for every $\sum_{s\in\Gamma}a_su_s\in C_c(\Gamma,A)$.

\begin{proof}~[Proof of Theorem~\ref{homeomorphism between states}.]\

By Lemma~\ref{ResInv}, the restriction map $R:S^1(A\rtimes\Gamma)\to S_\Gamma(A)$ given by $R(\fai)=\fai|_A$ for every $\fai\in S^1(A\rtimes\Gamma)$, is well-defined. Since $A$ is a $C^*$-subalgebra of $A\rtimes\Gamma$, the map $R$ is continuous under weak-$*$ topology.

If $R(\fai_1)=R(\fai_2)$ for $\fai_1,\fai_2\in S^1(A\rtimes\Gamma)$, then $\fai_1(a)=\fai_2(a)$ for all $a\in A$. By Lemma~\ref{ResInv},
$$\fai_1(au_s)=\fai_2(au_s),$$ for all $a\in A$ and $s\in\Gamma$. Since every element in $C_c(\Gamma,A)$ is a linear combination of $au_s$ and $C_c(\Gamma,A)$ is a dense subspace of $A\rtimes\Gamma$.  Hence $\fai_1=\fai_2$ and $R$ is injective.

Moreover given a $\fai\in S_\Gamma(A)$. By Lemma~\ref{covariant}, $\fai$ gives a representation $\rho_\fai$ of $A\rtimes\Gamma$ on $L^2(A,\fai)$. Let $\tilde{\fai}$ be the state of $A\rtimes\Gamma$ given by
$\tilde(\fai)(b)=\langle \rho_\fai(b)(\hat{1}),\hat{1}\rangle$ for all $b\in A\rtimes\Gamma$. By the definition of $\rho_\fai$, we see that $\tilde{\fai}\in S^1(A\rtimes\Gamma)$ and $\tilde{\fai}|_A=\fai$. This shows the surjectivity of $R$.

Above all $R$ is a bijective continuous map between two compact Hausdorff spaces $S^1(A\rtimes\Gamma)$ and $S_\Gamma(A)$. Therefore $R$ is a homeomorphism.

Note that $R$ is an affine map between two convex spaces $S^1(A\rtimes\Gamma)$ and $S_\Gamma(A)$, so the set of extreme points of $S^1(A\rtimes\Gamma)$ is homeomorphic to  $E_\Gamma(A)$.

Suppose that $\fai$ is an extreme point of $S^1(A\rtimes\Gamma)$ and $\fai=\lambda\fai_1+(1-\lambda)\fai_2$ for two states $\fai_1,\fai_2$ on $A\rtimes\Gamma$ and some $0<\lambda<1$. Then
$1=\fai(u_s)=\lambda\fai_1(u_s)+(1-\lambda)\fai_2(u_s)$ for every $s\in\Gamma$. It follows that $\fai_1(u_s)=\fai_2(u_s)=1$, that is, $\fai_1,\fai_2\in S^1(A\rtimes\Gamma)$. Hence $\fai=\fai_1=\fai_2$ and $\fai$ is a pure state.
\end{proof}

For a character $\xi$ on $\Gamma$~(a group homomorphism from $\Gamma$ to $\mathbb{T}$), denote $\{\fai\in S(A\rtimes\Gamma)|\fai(u_s)=\xi(s) \, \rm{for \,all} \,s\in\Gamma\}$ by $S^\xi(A\rtimes\Gamma)$ and  $\{\fai\in P(A\rtimes\Gamma)|\fai(u_s)=\xi(s) \, \rm{for \,all} \,s\in\Gamma\}$ by $P^\xi(A\rtimes\Gamma)$.

For a representation of $A\rtimes\Gamma$ on a Hilbert space $H$, define
$H_\xi=\{x\in H|\pi(u_s)(x)=\xi(s)x\, {\rm for \,all}\, s\in\Gamma\}$.

We have the following improvement of Theorem~\ref{homeomorphism between states}.
\begin{corollary}~\label{homeomorphism between states abelian}
Let $\xi$ be a character on $\Gamma$. When equipped with weak-$*$ topologies, $S_\Gamma(A)\cong S^\xi(A\rtimes\Gamma)$ and $E_\Gamma(A)\cong P^\xi(A\rtimes\Gamma)$.
\end{corollary}
\begin{proof}
By Theorem~\ref{homeomorphism between states}, $S_\Gamma(A)\cong S^1(A\rtimes\Gamma)$. Note that $S^1(A\rtimes\Gamma)\cong S^\xi(A\rtimes\Gamma)$~($P^1(A\rtimes\Gamma)\cong P^\xi(A\rtimes\Gamma)$ follows) and the homeomorphism is induced by the isomorphism
$\Lambda: A\rtimes\Gamma\to A\rtimes\Gamma$ given by $\Lambda(\sum_{s\in\Gamma} a_su_s)=\sum_{s\in\Gamma} \xi(s)a_su_s$ for all $\sum_{s\in\Gamma} a_su_s\in C_c(\Gamma,A)$.
\end{proof}

For a representation $\pi:A\rtimes\Gamma\to B(H)$, denote $\{x\in H|\pi(u_s)(x)=x\, {\rm for \,all}\, s\in\Gamma\}$ by $H_\Gamma$.
\begin{proposition}~\label{nonempty}
For a $C^*$-dynamical system $(A,\Gamma,\alpha)$, the following are equivalent.
\begin{enumerate}
\item  The set $S_\Gamma(A)$ is nonempty.
\item  The canonical homomorphism $C^*(\Gamma)\to A\rtimes\Gamma$ is an embedding.
\item There exists a representation $\pi: A\rtimes\Gamma\to B(H)$ such that $H_\Gamma\neq 0$, or equivalently, there exists a covariant representation $(\pi, U,H)$ of $(A,\Gamma,\alpha)$ such that $U$ contains the trivial representation of $\Gamma$.
\end{enumerate}
\end{proposition}
\begin{proof}
$(1)\Longrightarrow(2)$.

Take a $\fai\in S_\Gamma(A)$.
Let $\Gamma$ act on $\mathbb{C}$ trivially. By the invariance, the map $\fai:A\to\mathbb{C}$ is a $\Gamma$-equivariant contractive completely positive map. By~\cite[Exercise 4.1.4]{BrownOzawa2008},  there exists a contractive completely positive map $\tilde{\fai}: A\rtimes\Gamma\to C^*(\Gamma)$ such that $\tilde{\fai}(\sum_{s\in\Gamma} a_su_s)=\sum_{s\in\Gamma} \fai(a_s)u_s.$ Immediately one can check that the composition of maps $$C^*(\Gamma)\to A\rtimes\Gamma \underset{\tilde{\fai}}{\rightarrow} C^*(\Gamma)$$ is the identity map.  Hence the canonical homomorphism $C^*(\Gamma)\to A\rtimes\Gamma$ is an embedding.

$(2)\Longrightarrow(1)$.

By Theorem~\ref{homeomorphism between states}, it suffices to show $S^1(A\rtimes\Gamma)$ is nonempty.

Let $\pi_0:\Gamma\to\mathbb{C}$ be the trivial unitary representation of $\Gamma$ on $\mathbb{C}$. Then $\pi_0$ is a state on $C^*(\Gamma)$ such that $\pi_0(u_s)=1$ for every $s\in\Gamma$. Note that $C^*(\Gamma)$ is a Banach subspace of $A\rtimes\Gamma$. By the Hahn-Banach Theorem, one can extend $\pi_0$ to a bounded linear functional $\fai$ on $A\rtimes\Gamma$ without changing its norm~\cite[Corollary 6.5]{Conway1990}. Hence $\|\fai\|=\|\pi_0\|=1=\pi_0(1)=\fai(1)$. So $\fai$ is a state on $A\rtimes\Gamma$~\cite[Theorem 4.3.2]{KadisonRingrose1997I}, and satisfies that $\fai(u_s)=1$ for all $s\in\Gamma$.

$(1)\Longrightarrow(3)$.

This is guaranteed by Lemma~\ref{covariant}.

$(3)\Longrightarrow(1)$.

Take a unit vector $x$ in $H_\Gamma$ and define a state $\fai$ on $A\rtimes\Gamma$ by $\fai(b)=\langle\pi(b)x,x\rangle$ for all $b\in A\rtimes\Gamma$. It follows that $\fai\in S^1(A\rtimes\Gamma)$.
\end{proof}

\begin{remark}
The equivalence of (1) and (2) may be well-known.  When $A$ is commutative and $\Gamma$ is locally compact, this is mentioned in~\cite[Remark 7.5]{WillettYu2014}. To the best of our knowledge, it does not appear elsewhere in the literature.
\end{remark}

Notice that $\ub_\fai$ gives rise to an action of $\Gamma$ on $B(L^2(A,\fai))$, also denoted by $\alpha$  for convenience, defined by
$$\alpha_s(T)=\ub_\fai(s)T\ub_\fai(s^{-1})$$ for every $T\in B(L^2(A,\fai))$ and $s\in\Gamma$.

Denote $\langle T(\hat{1}),\hat{1}\rangle$ by $\fai(T)$ for all $T\in B(L^2(A,\fai))$. When $\fai$ is a $\Gamma$-invariant state on $A$, it is also a $\Gamma$-invariant state on $B(L^2(A,\fai))$ since
\begin{align*}
&\fai(\alpha_s(T))=\fai(\ub_\fai(s)T\ub_\fai(s^{-1}))=\langle \ub_\fai(s)T\ub_\fai(s^{-1})(\hat{1}),\hat{1}\rangle   \\
&=\langle T\ub_\fai(s^{-1})(\hat{1}),\ub_\fai(s^{-1})(\hat{1})\rangle=\langle T(\hat{1}),\hat{1}\rangle=\fai(T)
\end{align*}
for all $s\in\Gamma$.

We can also see that
\begin{equation}~\label{equivariant}
\alpha_s(\pi_\fai(a))=\pi_\fai(\alpha_s(a))
\end{equation}
for every $a\in A$ and $s\in\Gamma$.

We call a $T\in B(L^2(A,\fai))$ is {\bf $\Gamma$-invariant} if $\alpha_s(T)=T$ for all $s\in\Gamma$. Denote the set of $\Gamma$-invariant operators in $B(L^2(A,\fai))$ by $B(L^2(A,\fai))_\Gamma$. Let
$$\pi_\fai(A)'=\{T\in B(L^2(A,\fai))\,|\, T\pi_\fai(a)=\pi_\fai(a)T\, {\rm for\, all}\, a\in A\}.$$

\begin{proposition}~\label{ergodic via projections}
A $\Gamma$-invariant state $\fai$ on $A$ is ergodic  iff  $\fai(T^*T)=|\fai(T)|^2$ for every $T\in B(L^2(A,\fai))_\Gamma\cap \pi_\fai(A)'$.
\end{proposition}
\begin{proof}
Recall that $R:S^1(A\rtimes\Gamma)\to S_\Gamma(A)$ is the restriction map. For $\fai\in S_\Gamma(A)$, denote $R^{-1}(\fai)$ by $\psi$. By Theorem~\ref{homeomorphism between states}, $\psi$ is in $S^1(A\rtimes\Gamma)$.

Observe that $B(L^2(A,\fai))_\Gamma\cap \pi_\fai(A)'=\pi_\psi(A\rtimes\Gamma)'$.

Again by Theorem~\ref{homeomorphism between states}, the state $\fai$ on $A$ is an ergodic $\Gamma$-invariant state  iff $\psi$ is a pure state on $A\rtimes\Gamma$ iff $\pi_\psi(A\rtimes\Gamma)'=\mathbb{C}$. The ``only if'' part follows immediately.

Now suppose $\fai(T^*T)=|\fai(T)|^2$ for every $T\in B(L^2(A,\fai))_\Gamma\cap \pi_\fai(A)'=\pi_\psi(A\rtimes\Gamma)'$. A straightforward calculation shows that
$T(\hat{1})=\fai(T)\hat{1}$. Then for every $a\in A$, we have
\begin{align*}
&T(\hat{a})=T\pi_\fai(a)(\hat{1})=\pi_\fai(a)T(\hat{1})  \\
&=\pi_\fai(a)(\fai(T)\hat{1})=\fai(T)\hat{a}.
\end{align*}
This means $T=\fai(T)$, a scalar multiple of the identity operator. Hence $\pi_\psi(A\rtimes\Gamma)'=\mathbb{C}$ and $\psi$ is a pure state.
\end{proof}

\begin{remark}
The key observation $B(L^2(A,\fai))_\Gamma\cap \pi_\fai(A)'=\pi_\psi(A\rtimes\Gamma)'$ in the proof is pointed out to us by Sven Raum.
\end{remark}

\subsection{Ergodic $\Gamma$-invariant states on $A$ and irreducible representations of $A\rtimes\Gamma$.}\

We say a representation $\pi_1: B\to B(H_1)$ of a  $C^*$-algebra $B$ is {\bf unitarily equivalent} to a representation $\pi_2: B\to B(H_2)$ if there exists a surjective isometry $U:H_1\to H_2$ such that
$U\pi_1(b)(x)=\pi_2(b)U(x)$ for all $b\in B$ and $x\in H_1$.

\begin{theorem}~\label{ErgodicIrr}
A representation  $\pi: A\rtimes\Gamma\to B(H)$ is unitarily equivalent to $\rho_\fai: A\rtimes\Gamma\to B(L^2(A,\fai))$ for some  $\fai\in E_\Gamma(A)$ iff $\pi$ is irreducible and  $H_\Gamma\neq 0$.
\end{theorem}

\begin{proof}

For a $\fai\in E_\Gamma(A)$, by Theorem~\ref{homeomorphism between states}, there exists a $\psi\in P^1(A\rtimes\Gamma)$ such that $R(\psi)=\fai$.

Next we show that $\rho_\fai$ is unitarily equivalent to the GNS representation $\pi_\psi:A\rtimes\Gamma\to B(L^2(A\rtimes\Gamma,\psi))$ of $A\rtimes\Gamma$ with respect to $\psi$. Since  $\psi$ is a pure state, this shows $\rho_\fai$ is irreducible.
\begin{claim*}
$\rho_\fai$ is unitarily equivalent to $\pi_\psi$.
\end{claim*}
\begin{proof}
Define $\Phi:L^2(A\rtimes\Gamma,\psi)\to L^2(A,\fai)$ by
$$\Phi((\sum_{s\in\Gamma}f_su_s)^\wedge)=\sum_{s\in\Gamma}f_s+I_\fai.$$ Here $I_\psi=\{b\in A\rtimes\Gamma|\psi(b^*b)=0\}$ and $I_\fai=\{a\in A|\fai(a^*a)=0\}$.

 Now we check $\Phi$ is a Hilbert space isomorphism.

By Lemma~\ref{ResInv} we have
\begin{align*}
&\langle(\sum_{s\in\Gamma}f_su_s)^\wedge,(\sum_{t\in\Gamma}g_tu_t)^\wedge\rangle=\psi((\sum_{t\in\Gamma}g_tu_t)^*(\sum_{s\in\Gamma}f_su_s))   \\
&=\psi(\sum_{s,t\in\Gamma}u_{t^{-1}}g_t^*f_su_s)=\psi(\sum_{s,t\in\Gamma}g_t^*f_s)    \\
&=\fai((\sum_{t\in\Gamma}g_t)^*(\sum_{s\in\Gamma}f_s))=\langle(\sum_{s\in\Gamma}f_s)^\wedge,(\sum_{t\in\Gamma}g_t)^\wedge\rangle.
\end{align*}
Thus $\Phi$ is an isometry. The image of $\Phi$ is dense in $L^2(A,\fai)$, so $\Phi$ is also surjective. These prove that $\Phi$ is an isomorphism between Hilbert spaces.

Next we verify unitary equivalence between $\pi_\psi$ and $\rho_\fai$.
\begin{align*}
&\Phi\pi_\psi(\sum_{s\in\Gamma}f_su_s)((\sum_{t\in\Gamma}g_tu_t)^\wedge)=\Phi((\sum_{s,t\in\Gamma}f_su_sg_tu_t)^\wedge) \\
&=\Phi((\sum_{s,t\in\Gamma}f_su_sg_t u_{s^{-1}}u_{st})^\wedge)=(\sum_{s,t\in\Gamma}f_s\alpha_s(g_t))^\wedge.
\end{align*}
On the other hand,
\begin{align*}
&\rho_\fai((\sum_{s\in\Gamma}f_su_s)\Phi(\sum_{t\in\Gamma}g_tu_t)^\wedge)=\rho_\fai(\sum_{s\in\Gamma}f_su_s)((\sum_{t\in\Gamma}g_t)^\wedge) \\
&=(\sum_{s,t\in\Gamma}f_s\alpha_s(g_t))^\wedge.
\end{align*}
So $\pi_\psi(b)(x)=\Phi^{-1}\rho_\fai(b)\Phi(x)$ for all $b\in A\rtimes\Gamma$ and $x\in L^2(A\rtimes\Gamma,\psi)$. Hence $\pi_\psi$ and $\rho_\fai$ are unitarily equivalent.
\end{proof}
Note that $0\neq \hat{1}\in L^2(A,\fai)$ and $\rho_\fai(u_s)(\hat{1})=\hat{1}$ for all $s\in\Gamma$. Hence $L^2(A,\fai)_\Gamma\neq 0$.

Conversely given an  irreducible representation $\pi: A\rtimes\Gamma\to B(H)$  with $H_\Gamma\neq 0$, take a unit vector $x\in H_\Gamma$ and define a state $\psi$ on $A\rtimes\Gamma$ by
$$\psi(b)=\langle \pi(b)x,x\rangle $$ for all $b\in A\rtimes\Gamma$. Since $\pi$ is irreducible, the state $\psi$ is a pure state and the GNS representation of $A\rtimes\Gamma$ with respect to $\psi$, $\pi_\psi$  is unitarily equivalent to $\pi$~\cite[Theorem I.9.8]{Davidson1996}~\cite[2.4.6]{Dixmier1977}. Also $x\in H_\Gamma$ implies that $\psi(u_s)=1$  for all $s\in\Gamma$. So $\fai=\psi|_A\in E_\Gamma(A)$. By the previous claim  $\pi_\psi$ is unitarily equivalent to $\rho_\fai$. This finishes the proof.
\end{proof}

Now we consider the case when $A$ is commutative.
\begin{theorem}~\label{dim}
For any irreducible representation $\pi: C(X)\rtimes\Gamma\to B(H)$, we have $\dim{H_\Gamma}\leq1$. If $H_\Gamma\neq 0$, then there exists a unique ergodic $\Gamma$-invariant state $\fai$ on $C(X)$~(or a unique regular Borel probability measure on $X$) such that $\pi$ is unitarily equivalent to $\rho_\fai$.
\end{theorem}
\begin{proof}
Suppose $H_\Gamma\neq 0$ for an irreducible representation $\pi: C(X)\rtimes\Gamma\to B(H)$.

Take unit vectors $x,y\in H_\Gamma$. Define a state $\psi$ on $A\rtimes\Gamma$ by
$$\psi(b)=\langle \pi(b)x,x\rangle $$ for all $b\in C(X)\rtimes\Gamma$. Then $\fai=\psi|_{C(X)}$ gives an ergodic $\Gamma$-invariant probability measure $\mu$ on $X$ with $\fai(f)=\int_X f\,d\mu$ for all $f\in C(X)$. Also the GNS representation $\pi_\psi$ of $C(X)\rtimes\Gamma$ with respect to $\psi$,  is unitarily equivalent to $\rho_\fai: C(X)\rtimes\Gamma\to B(L^2(A,\fai))$. Note that $L^2(A,\fai)=L^2(X,\mu)$ and $L^2(A,\fai)_\Gamma$ consists of $\Gamma$-invariant functions in $L^2(X,\mu)$, which are always constant functions~\cite[Chapter3, 3.10.]{Glasner2003}. Under  surjective isometries  $H\cong L^2(A\rtimes\Gamma,\psi)\cong L^2(X,\mu)$, both $x$ and $y$ are mapped to $\Gamma$-invariant functions in $L^2(X,\mu)$. Since $\mu$ is ergodic, their images in $L^2(X,\mu)$ are both constant functions. Hence there exists a constant $\lambda$ with absolute value 1 such that $x=\lambda y$. This shows that $\dim{H_\Gamma}=1$.

For the second part, the existence of $\fai$ follows from Theorem~\ref{ErgodicIrr}.

To prove the uniqueness of $\fai$, we show the following.
\begin{claim*}
If $\rho_\fai\sim\rho_\psi$ for $\fai,\psi\in S_\Gamma(C(X))$, then $\fai=\psi$.
\end{claim*}
\begin{proof}
Let $\Theta: L^2(A,\fai)\to L^2(A,\psi)$ be an isomorphism. It is easy to see that $\Theta$ preserves $\Gamma$-invariant vectors, i.e., $\Theta: L^2(A,\fai)_\Gamma\to L^2(A,\psi)_\Gamma$ is also an isomorphism. Hence $\Theta(\hat{1})=\lambda\hat{1}$ for some complex number $\lambda$ with $|\lambda|=1$.

By definition of $\rho_\fai$, we have $\fai(f)=\langle \rho_\fai(f)\hat{1},\hat{1}\rangle$ for all $f\in C(X)$. It follows that
$$\psi(f)=\langle \rho_\psi(f)\hat{1},\hat{1}\rangle=\langle \Theta^{-1}\rho_\fai(f)\Theta\hat{1},\hat{1}\rangle=\langle \rho_\fai(f)\lambda\hat{1},\lambda\hat{1}\rangle=\fai(f)$$ for all $f\in C(X)$.
\end{proof}

Hence $\fai$ is uniquely determined by the unitary equivalence class of $\pi$.
\end{proof}

\begin{remark}
\begin{enumerate}
\item Theorem~\ref{ErgodicIrr} and Theorem~\ref{dim} say that classification of ergodic $\Gamma$-invariant regular Borel probability measures on a compact Hausdorff space $X$ amounts to classification of equivalence classes of irreducible representations of $C(X)\rtimes\Gamma$ whose restriction to $\Gamma$ contains trivial representation.
\item When $A$ is non-commutative, Theorem~\ref{dim} fails. For instance, one can take a noncommutative $C^*$-algebra $A$ and a discrete group $\Gamma$ acting on $A$ trivially. An irreducible representation
$\pi: A\to B(H)$ with $\dim{H}>1$ and the trivial representation $\Gamma\to B(H)$ give rise to an irreducible representation of $\rho:A\rtimes\Gamma\to B(H)$. But $H=H_\Gamma$ is not of dimension 1. 
\end{enumerate}
\end{remark}

There is an immediate application of Theorem~\ref{dim} to representation theory of semidirect product groups.
\begin{corollary}~\label{rep}
Suppose a discrete group $\Gamma$ acts on a discrete abelian group $G$ by group automorphisms. Every irreducible unitary representation $\pi:G\rtimes \Gamma\to B(H)$ of the semidirect product group $G\rtimes \Gamma$ satisfies that $\dim{H_\Gamma}\leq 1$. When $\dim{H_\Gamma}=1$, the representation $\pi$ is induced by an ergodic $\Gamma$-invariant regular Borel probability measure $\mu$ on the Pontryagin dual $\widehat{G}$ of $G$.
\end{corollary}
\begin{proof}
Note that $\Gamma$ acts on group $C^*$-algebra $C^*(G)$ as automorphisms, $C^*(G)=C(\hat{G})$ for the dual group $\hat{G}$ of $G$ and $C^*(G)\rtimes\Gamma\cong C^*(G\rtimes\Gamma)$. There exists a 1-1 correspondence between  irreducible unitary representations of $G\rtimes\Gamma$ and irreducible representations of $C^*(G\rtimes\Gamma)$. Apply Theorem~\ref{dim} to the case $C(X)=C(\hat{G})$.
\end{proof}

\section{Furstenberg's $\times p,\times q$ problem via representation theory}\

Let  $S,T:X\to X$ be two commuting continuous maps on a compact Hausdorff space $X$. A Borel probability measure $\mu$ on $X$ is called {\bf $S,T$-invariant} if $\mu(S^{-1}A)=\mu(T^{-1}A)=\mu(A)$ for every Borel subset $A$ of $X$. An $S,T$-invariant measure $\mu$ is called {\bf ergodic} if every Borel set $E$ with $S^{-1}E=E=T^{-1}E$ satisfies that $\mu(E)=0$ or 1.

Define  maps $T_p, T_q:\mathbb{T}\to \mathbb{T}$ as $T_p(z)=z^p$ and $T_q(z)=z^q$ for all $z\in\mathbb{T}$.

A Borel probability  measure $\mu$ on $\mathbb{T}$ is called {\bf $\times p,\times q$-invariant} if it is $T_p,T_q$-invariant. A Borel set $E\subset\mathbb{T}$ is called $\times p,\times q$-invariant if $A=T_p A=T_q A$.

We can define $\times p,\times q$ maps $T_p, T_q$ on $\mathbb{Z}[\frac{1}{pq}]$ by $T_p(g)=pg, T_q(g)=qg$  for every $g\in \mathbb{Z}[\frac{1}{pq}]$. Note that $T_p$ and $T_q$ are group automorphisms. Hence they induces group automorphisms on the dual group $S_{pq}$ of $\mathbb{Z}[\frac{1}{pq}]$. For convenience we also call them $\times p,\times q$ maps on $S_{pq}$.

Denote the set of $\times p,\,\times q$-invariant measures on  unit circle by $M_{p,q}(\mathbb{T})$, the set of ergodic $\times p,\,\times q$-invariant measures on unit circle by $EM_{p,q}(\mathbb{T})$, the set of $\times p,\,\times q$-invariant measures on  $S_{pq}$ by $M_{p,q}(S_{pq})$, the set of ergodic $\times p,\,\times q$-invariant measures on  $S_{pq}$  by $EM_{p,q}(S_{pq})$.

\subsection{$\times p,\times q$-invariant measures on $pq$-solenoid and $\times p,\times q$-invariant measures on the unit circle}\

The following result is well-known for experts.  For completeness we give a proof here.

\begin{proposition}~\label{invariant measures on solenoid}
When equipped with weak-$*$ topologies, the restriction map $R: M_{p,q}(S_{pq})\to M_{p,q}(\mathbb{T})$ defined by $R(\mu)(f)=\mu(f)$ for $\mu\in M_{p,q}(S_{pq})$ and $f\in C(\mathbb{T})$ is a homeomorphism. Also $R$ restricts a homeomorphism from $EM_{p,q}(S_{pq})$ to $EM_{p,q}(\mathbb{T})$.
\end{proposition}
\begin{proof}
Take $\mu\in  M_{p,q}(S_{pq})$. Since $C(\mathbb{T})$ is a $C^*$-subalgebra of $C(S_{pq})$, the restriction $R(\mu)$ of $\mu$ on $C(\mathbb{T})$ belongs to  $M_{p,q}(\mathbb{T})$ and $R$ is also continuous under the weak-$*$ topology.

Conversely, assume that $\mu\in M_{p,q}(\mathbb{T})$. Note that the group algebra $\mathbb{C}\mathbb{Z}([\dfrac{1}{pq}])$ is a dense $*$-subalgebra of $C(S_{pq})$ and define $\nu(z^{kp^mq^n})=\mu(z^k)$ for  $n,m,k\in\mathbb{Z}$. By Bochner's Theorem~\cite[1.4.3]{Rudin1990} $\nu$ is a Borel probability measure on $S_{pq}$ iff $\{\nu(z^k)\}_{k\in\mathbb{Z}[\frac{1}{pq}]}$ is a positive definite sequence.

For any finite subset $F$ of $\mathbb{Z}[\frac{1}{pq}]$, there exist positive integers $k,l$ such that $F'=p^kq^lF=\{p^kq^ls|s\in F\}$ is a finite subset of $\mathbb{Z}$. Then we have
\begin{align*}
&\nu((\sum_{s\in F}\lambda_s z^s)^*(\sum_{t\in F}\lambda_t z^t))=\sum_{s,t\in F}\bar{\lambda}_s\lambda_t \nu(z^{t-s})  \\
&=\sum_{s,t\in F}\bar{\lambda}_s\lambda_t \mu(z^{p^kq^l(t-s)})=\mu((\sum_{s\in F'}\lambda_{sp^{-k}q^{-l}} z^s)^*(\sum_{t\in F'}\lambda_{tp^{-k}q^{-l}} z^t))\geqslant 0.        \\
\end{align*}
Furthermore the $\times p,\,\times q$-invariance of $\nu$ follows from the definition. This shows that $\nu$ is in $M_{p,q}(S_{pq})$. Moreover $\mu(z^k)=\nu(z^k)$ for all $k\in\mathbb{Z}$ by the $\times p,\times q$-invariance of $\mu$, hence $\mu$ is the restriction of $\nu$ on $C(\mathbb{T})$, and this proves the subjectivity of $R$.

On the other hand, if $R(\mu_1)=R(\mu_2)$ for $\mu_1,\mu_2\in M_{p,q}(S_{pq})$, then $\mu_1(z^k)=\mu_2(z^k)$ for all $k\in\mathbb{Z}$. Since $\mu_1$  and $\mu_2$ are $\times p,\times q$-invariant, we have $\mu_1(z^{kp^mq^n})=\mu_2(z^{kp^mq^n})$ for all $n,m,k\in\mathbb{Z}$. This proves the injectivity of $R$.

So $R$ is a bijective continuous map between two compact Hausdorff spaces $M_{p,q}(S_{pq})$ and $M_{p,q}(\mathbb{T})$, this implies that $R$ is a homeomorphism.

Furthermore $R$ is a homeomorphism from $EM_{p,q}(S_{pq})$ to $EM_{p,q}(\mathbb{T})$ since $R$ is  affine.
\end{proof}

\begin{theorem}~\label{finiterepresentation}
A representation $\pi:C(S_{pq})\rtimes \mathbb{Z}^2\to B(H)$ is induced by a finitely supported ergodic $\times p,\times q$-invariant measure $\mu$ on $\mathbb{T}$ if and only if
\begin{enumerate}
\item $\pi$ is irreducible;
\item $H_{\mathbb{Z}^2}\neq 0$;
\item There exists nonzero $N\in\mathbb{Z}\subset\mathbb{Z}[\frac{1}{pq}]$ such that $\pi(z^N)x=x$ for every $x\in H_{\mathbb{Z}^2}$.
\end{enumerate}
\end{theorem}
\begin{proof}
Suppose that $\pi:C(S_{pq})\rtimes \mathbb{Z}^2\to B(H)$ is induced by a finitely supported ergodic $\times p,\times q$-invariant measure $\mu$ on $\mathbb{T}$.

Since both $\times p$ and $\times q$ maps have zero entropy with respect to $\mu$, there is a representation $\pi_\mu: C(S_{pq})\rtimes \mathbb{Z}^2\to B(L^2(\mathbb{T},\mu))$ induced by $\mu$~(see Introduction for the definition of $\pi_\mu$).

By Proposition~\ref{invariant measures on solenoid}, $\nu=R^{-1}(\mu)$ is an ergodic $\times p,\times q$-invariant measure on $S_{pq}$. Hence the representation $\rho_\nu$ of $C(S_{pq})\rtimes \mathbb{Z}^2$ on $L^2(S_{pq},\nu)$ is irreducible.

Note that $L^2(\mathbb{T},\mu)$ is a subspace of $L^2(S_{pq},\nu)$ since $\mu$ is the restriction of $\nu$ from  $C(S_{pq})$ onto $C(\mathbb{T})$. Also $L^2(\mathbb{T},\mu)$ is a nonzero invariant subspace of $L^2(S_{pq},\nu)$ under $\rho_\nu$. Hence $L^2(\mathbb{T},\mu)=L^2(S_{pq},\nu)$. Hence $\pi_\mu$ is unitarily equivalent to $\rho_\nu$. So $\pi$ is irreducible and $H_{\mathbb{Z}^2}\neq 0$.

Moreover $\mu$ is finitely supported in a subset of $\{\frac{i}{N}\}_{i=0}^{N-1}\subset[0,1)$~(here we identify $\mathbb{T}$ with $[0,1)$). Hence $\mu(z^N)=1$ which implies that  $\pi(z^N)x=x$ for every $x\in H_{\mathbb{Z}^2}$.

Conversely assume that $\pi$ is an irreducible representation satisfying that $H_{\mathbb{Z}^2}\neq 0$ and $\pi(z^N)x=x$ for a nonzero $N\in\mathbb{Z}$ and every $x\in H_{\mathbb{Z}^2}$.

\begin{claim*}
$H$ is finite dimensional.
\end{claim*}
\begin{proof}
Take a unit $y\in H_{\mathbb{Z}^2}$. Then $\overline{{\rm Span}\pi(\mathbb{Z}[\frac{1}{pq}])y}$ is an invariant subspace of $H$ under $\pi$. Since $\pi$ is irreducible, we have
$$H=\overline{{\rm Span}(\pi(\mathbb{Z}[\frac{1}{pq}])y)}.$$

So it suffices to prove ${\rm Span}(\pi(\mathbb{Z}[\frac{1}{pq}])y)$ is finite dimensional.

Firstly we prove that $\pi(z^M)y=y$ for a positive integer $M$ coprime to $pq$.

Without loss of generality we can assume $N>0$.  There exist nonnegative integers $i,j,K,M$ such that $KN=Mp^iq^j$ with $M$ coprime to $pq$. Then $\pi(z^{Mp^iq^j})y=\pi(z^{KN})y=y$.

Note that $\mathbb{Z}^2$ acts on $\mathbb{Z}[\frac{1}{pq}]$ by $\times p,\times q$, that is, $(m,n)\cdot z^k=z^{kp^mq^n}$ for all $m,n\in\mathbb{Z}$ and every $k\in \mathbb{Z}[\frac{1}{pq}]$. Since $y$ is in $H_{\mathbb{Z}^2}$,

we have $\pi((i,j))\pi(z^M)y=\pi(z^{Mp^iq^j})\pi((i,j))y=\pi(z^{Mp^iq^j})y=y$, which implies $\pi(z^M)y=y$.

Secondly we prove that $\pi(\mathbb{Z}[\frac{1}{pq}])y=\pi(\mathbb{Z})y$.

For all nonnegative integers $i,j$, there exists an integer $l$ such that $lp^iq^j=rM+1$ since $M$ is coprime to $pq$. Hence for every positive integer $k$, we have
\begin{align*}
&\pi(z^{\frac{k}{p^iq^j}})y=\pi(z^{\frac{k(lp^iq^j-rM)}{p^iq^j}})y=\pi(z^{kl})\pi(z^{\frac{-krM}{p^iq^j}})y  \\
\tag{$\pi(z^M)y=y$\, and\, $y\in H_{\mathbb{Z}^2}$}
&=\pi(z^{kl})y.
\end{align*}
This shows that  $\pi(\mathbb{Z}[\frac{1}{pq}])y\subseteq\pi(\mathbb{Z})y$.

Lastly we prove that ${\rm Span}(\pi(\mathbb{Z})y)$ is finite dimensional.

Every $k\in\mathbb{Z}$ can be written as $k=lN+r$ for some $l\in\mathbb{Z}$ and $0\leq r<N$. Hence $\pi(z^k)y=\pi(z^{lN+r})y=\pi(z^r)y$. This implies that
$$\pi(\mathbb{Z})y\subset {\rm Span}\{\pi(z^i)y\}_{i=0}^{N-1}.$$ So $\dim{H}\leq N$. We finish proof of the claim.
\end{proof}

Define a state $\psi$ on $C(S_{pq})\rtimes \mathbb{Z}^2$ by $\psi(b)=\langle \pi(b)y,y\rangle$ for every $b\in C(S_{pq})\rtimes \mathbb{Z}^2$. We have $\psi\in P^1(C(S_{pq})\rtimes \mathbb{Z}^2)$ since $\pi$ is irreducible and $y\in H_{\mathbb{Z}^2}$. By Theorem~\ref{homeomorphism between states}, $\nu=R(\psi)=\psi|_{C(S_{pq})}$ is an ergodic $\times p,\times q$-invariant measure on $S_{pq}$.

By Theorem~\ref{ErgodicIrr} and Theorem~\ref{dim}, we have $\pi\sim\rho_\nu$. Of course $H\cong L^2(S_{pq},\nu)$. From the claim,  $L^2(S_{pq},\nu)$ is finite dimensional.

Hence $\nu$ is finitely supported in $S_{pq}$. Let  $\mu=R(\nu)=\nu|_{C(S_{pq})}$. As before, $\rho_\nu$ is unitarily equivalent to $\pi_\mu$. Hence $L^2(\mathbb{T},\mu)\sim L^2(S_{pq},\nu)$ is finite dimensional. Hence $\mu$ is a finitely supported ergodic $\times p,\times q$-invariant measure on $\mathbb{T}$ and $\pi\sim \pi_\mu$.
\end{proof}

Consequently we have the following.

\begin{corollary}~\label{FviaR}
Furstenberg's conjecture is true iff there is a unique irreducible unitary representation $U:\mathbb{Z}[\frac{1}{pq}]\rtimes\mathbb{Z}^2\to B(H)$ such that $H_{\mathbb{Z}^2}\neq 0$ and $\pi(z^k)x\neq x$ for every nonzero integer $k$ and nonzero $x\in H_{\mathbb{Z}^2}$.
\end{corollary}

\end{document}